\documentclass[11pt]{amsart}
\usepackage{amssymb,amscd,latexsym}

\newtheorem{theorem}{Theorem}[section]
\newtheorem{lemma}[theorem]{Lemma}
\newtheorem{corollary}[theorem]{Corollary}
\newtheorem{proposition}[theorem]{Proposition}
\newtheorem{remark}[theorem]{Remark}

\newtheorem{definition}[theorem]{Definition}

\newcommand{\rar}{\rightarrow}
\newcommand{\lar}{\longrightarrow}

\def\kx{K[x_1,\ldots, x_n]}
\def\mgfr{minimal graded free resolution}
\def\iff{if and only if}
\def\N{\mathbb N}
\def\Q{\mathbb Q}
\def\ff{\mathbf f}
\def\iffm{\Leftrightarrow}
\def\depth#1{{\rm depth}\,(#1)}
\def\lcm#1{{\rm lcm}\,(#1)}
\def\comp#1{{\rm comp}\,(#1)}
\def\link#1{{\rm link}\,(#1)}

\begin{document}

\title[Nonlinear syzygies of ideals associated to graphs]
{First nonlinear syzygies of ideals associated to graphs}

\author[O. Fern\'andez-Ramos]{Oscar Fern\'andez-Ramos}
\email{oscarf@agt.uva.es}

\author[P. Gimenez]{Philippe Gimenez}
\email{pgimenez@agt.uva.es}

\address{Depto. de Algebra, Geometr\'{\i}a y
Topolog\'{\i}a, Facultad de Ciencias, Universidad de Valladolid, 47005
Valladolid, Spain}

\medskip

\begin{abstract}
Consider an ideal $I\subset \kx$, with $K$ an arbitrary field,
generated by monomials of degree two. Assuming that $I$ does not
have a linear resolution, we determine the step $s$ of the \mgfr{}
of $I$ where nonlinear syzygies first appear, we show that at this
step of the resolution nonlinear syzygies are concentrated in degree
$s+3$, and we compute the corresponding graded Betti number
$\beta_{s,s+3}$. The multidegrees of these nonlinear syzygies are
also determined and the corresponding multigraded Betti numbers are
shown to be all equal to 1.
\end{abstract}

\maketitle

\section{Introduction}\label{section_intro}

Let $R:=\kx$ be the polynomial ring in $n$ variables over an
arbitrary field $K$, and let $I$ be an ideal in $R$ generated by a
finite set $\ff=\{f_1,\ldots,f_m\}$ of distinct monomials in $R$ of
degree two. Associated to $I$ there is a graph $G(I)$ with vertex
and edge sets $ V_{G(I)}:=\{(1),\ldots,(n)\}$ and
$E_{G(I)}:=\{(i,j);\, 1\leq i\leq j\leq n\,/\ x_ix_j\in I\}$,
respectively, and we shall say that $I$ is an {\it ideal associated
to a graph}. There is another graph related to the ideal $I$ that
will be featured in this paper, namely, the ({\it simple}) {\it
complement} $G(I)^c$ of $G(I)$ with vertex set $V_{G(I)}$ and edge
set $E_{G(I)^c}:=\{(i,j);\, 1\leq i<j\leq n\,/\ (i,j)\notin
E_{G(I)}\}$. Note that the graph $G(I)$ may have loops while
$G(I)^c$ is always a simple graph. When $I$ is squarefree, i.e.,
when $G(I)$ is simple, $I$ is called an {\it edge ideal}. In this
connection we mention the survey in Villarreal's book
\cite[Chapter~6]{VillaBook}.

\smallskip
As a general aim, one would like to establish a correspondence
between algebraic properties of the ideal $I$ (or the ring $R/I$, or
the subalgebra $K[\ff]\subset R$) and the graph theoretical data of
$G(I)$ (or of any other graph associated to $I$ such as $G(I)^c$).
Two illustrations of this correspondence are the combinatorial
characterizations of normality and polarizability of the subalgebra
$K[\ff]\subset R$ given in \cite{SimVaVilla} and \cite{PolarSyz},
respectively. In this note, we shall focus on properties of $I$
related to its \mgfr{}
\begin{equation}\label{eq_resol}
0\rar \bigoplus_j R(-j)^{\beta_{p,j}}\lar\cdots\lar \bigoplus_j
R(-j)^{\beta_{0,j}} \lar I\rar 0
\end{equation}
where $R(-j)$ denotes the graded $R$-module obtained by shifting the
degrees of homogeneous elements in $R$ by $j$. The number of
generators of degree $j$ in the $i$th syzygy module is denoted by
$\beta_{i,j}$. These numbers do not depend on the \mgfr{} and are
called the {\it graded Betti numbers} of $I$. The numerical
information contained in the \mgfr{} of $I$ (degrees of the syzygies
and graded Betti numbers) can be displayed on a table with $p+1$
columns labeled $0,1,\ldots,p$ that correspond to the steps in the
resolution and with rows corresponding to degrees, where the entry
in the $j$th row of the $i$th column is $\beta_{i,i+j}$. This table
is called the {\it Betti diagram} of $I$; see \cite{Macaulay} and
\cite[page~7]{EisNew}. The size of the Betti diagram of $I$ is a
measure of the size of the \mgfr{} of $I$ and hence, in some sense,
a measure of the complexity of the ideal $I$. By
\cite[Proposition~1.9]{EisNew}, the first row (with a nonzero entry)
of the Betti diagram is the one labeled by the {\it initial degree}
$\delta$ of $I$, i.e., the smallest integer $j$ such that
$\beta_{0,j}\neq 0$. In our situation, $\delta=2$ and in fact
$\beta_{0,j}=0$ for all $j\neq 2$. The last row (with a nonzero
entry) of the diagram is the one labeled by $m$, the {\it
Castelnuovo-Mumford regularity} of $I$. The number of rows of the
Betti diagram is thus $m-\delta+1$, and the number of columns is
related to the depth of $R/I$ by the Auslander-Buchsbaum formula,
$p+1=n-\depth{R/I}$.

\smallskip
We say that $I$ has a {\it linear resolution} if $\beta_{i,j}=0$
whenever $j\neq i+2$. In terms of the Betti diagram, $I$ has a
linear resolution \iff{} the diagram has only one row. The classical
result of Fr\"oberg \cite[Theorem~1]{Froberg} characterizes edge
ideals with linear resolution. It has recently been recovered by
Eisenbud, Green, Hulek and Popescu that provide some additional
information. When the resolution of $I$ is nonlinear, they show in
\cite[Theorem~2.1]{EisEtal} that the step of the \mgfr{} of $I$
where nonlinear syzygies first appear is determined by the length of
the shortest cycle in $G(I)^c$ having no chord. A similar result for
ideals associated to graphs that are not squarefree can be deduced
from \cite[Proposition~2.3]{EisEtal}. In this note, we prove the
following more precise result:

\begin{theorem}\label{thm_main}
Consider an ideal $I\subset \kx$ generated by monomials of degree
two, with $K$ an arbitrary field. Let $G(I)$ be the graph associated
to $I$ and let $G(I)^c$ be its complement. For all $d\geq 3$, set
$\displaystyle{i_d:=\min_{1\leq i\leq p}{\{i;\ \beta_{i,i+d}\neq
0\}}}$ if $\beta_{i,i+d}\neq 0$ for some $i\geq 1$, $i_d:=0$
otherwise. Then\,{\rm:}
\begin{enumerate}
\item\label{item_in_mainthm_linpres}
$i_3=1$ \iff{} $G(I)$ has at least one induced subgraph consisting
of two disjoint edges. When this occurs, $\beta_{1,4}$ is the number
of such induced subgraphs of $G(I)$, and $\beta_{1,j}=0$ for all
$j>4$.
\item\label{item_in_mainthm_nonlinpres}
$i_3>1$ \iff{} $G(I)$ has no induced subgraph consisting of two
disjoint edges and $G(I)^c$ has at least one induced cycle of length
$\geq 5$. When this occurs, $i_3=r-3$ where $r$ is the smallest
integer $\geq 5$ such that $G(I)^c$ has an induced $r$-cycle,
$\beta_{i_3,i_3+3}$ is the number of induced $r$-cycles in $G(I)^c$,
and $\beta_{i,j}=0$ for all $i\leq i_3$ and $j>i_3+3$.
\item\label{item_in_mainthm_linres}
$i_3=0$ \iff{} $G(I)$ has no induced subgraph consisting of two
disjoint edges and $G(I)^c$ is chordal. When this occurs, $I$ has a
linear resolution.
\end{enumerate}
\end{theorem}

Note that when $I$ does not have a linear resolution, the step $i_3$
of the \mgfr{} of $I$ where nonlinear syzygies first appear is
determined, and the results of Eisenbud, Green, Hulek and Popescu
mentioned before are recovered. Moreover, we show that the nonlinear
syzygies at the $i_3$th step of the resolution are concentrated in
degree $i_3+3$ and we compute the corresponding Betti number
$\beta_{i_3,i_3+3}$.

\smallskip
If we consider the standard $\N^n$-grading on the polynomial ring
$R$, monomial ideals are $\N^n$-graded modules. In particular, $I$
has a minimal multigraded free resolution as the one in
(\ref{eq_resol}) where one substitutes multidegrees
$\underline{s}\in\N^n$ for degrees $j\in\N$. The number of syzygies
of multidegree $\underline{s}$ in the $i$th syzygy module is denoted
by $\beta_{i,\underline{s}}$ and these numbers are called the {\it
multigraded Betti numbers} of $I$. The multigraded resolution
provides a finer numerical information than the graded one. In
particular, the graded Betti numbers are recovered from the
multigraded ones by
 $\beta_{i,j}=\sum_{\underline{s};\,\vert
 \underline{s}\vert=j}\beta_{i,\underline{s}}$.
In this note, we shall work with the multigraded resolution and we
shall prove, in fact, a result which is stronger than the one stated
in Theorem~\ref{thm_main}: indeed, the multidegrees $\underline{s}$
such that $\beta_{i_3,\underline{s}}\neq 0$ and
$\vert\underline{s}\vert=i_3+3$ are determined, and we shall see
that $\beta_{i_3,\underline{s}}=1$ for each of these multidegrees.

\smallskip
The \mgfr{} of an ideal associated to a graph and, in particular,
its Betti numbers depend in general on the characteristic of the
field $K$; see \cite[Section~4]{Katzman}. Nevertheless, whenever a
property of the resolution can be characterized in terms of the
combinatorial data of the graph $G(I)$ or any other graph associated
to the generating monomials of $I$, this will not occur. The
aforementioned result of Fr\"oberg illustrates this observation. So
does our Theorem~\ref{thm_main} where the index $i_3$ and the
corresponding Betti number $\beta_{i_3,i_3+3}$ are computed
independently of the characteristic of $K$ . This happens with other
Betti numbers as observed by several authors in the case where $I$
is an edge ideal. For example, there is a formula for $\beta_{1,3}$
in \cite{ElVilla}; see \cite[Proposition~6.6.3]{VillaBook}. This
result is later recovered in \cite{RothVantuyl} where the first row
of the Betti diagram of edge ideals is studied. More precisely, a
closed formula for $\beta_{i-2,i}$ is given for $i\leq 5$, and for
all $i$ when $G(I)$ has no induced $4$-cycles. In \cite{Katzman}, it
is shown that for all $i\geq 0$, $\beta_{i,2(i+1)}$ is the number of
induced subgraphs of $G(I)$ consisting of $i+1$ disjoint edges. We
shall see in Lemma~\ref{lem_katzman_with_loops} that the same result
holds for arbitrary ideals associated to graphs.

\smallskip
In Section~\ref{section_froberg}, we recall Fr\"oberg's
characterization of edge ideals having a linear resolution and
review some refinements obtained recently. In particular, the
results of Eisenbud, Green, Hulek and Popescu cited before are
mentioned. In Section~\ref{section_example}, we focus on a special
class of edge ideals whose resolution is nonlinear. The Betti
diagram of these ideals is completely determined in
Proposition~\ref{prop_example} and we can observe that
Theorem~\ref{thm_main} holds in this case. This example plays an
important role in the proof of the main result in
Section~\ref{section_results}, Theorem~\ref{thm_lastthm}, that
computes the number of nonlinear syzygies of smallest degree of
ideals associated to graphs. In Section~\ref{section_shape}, a
description of the shape of the Betti diagram is given in
Theorem~\ref{thm_shape_Betti}. This is the last ingredient for the
proof of Theorem~\ref{thm_main} which is given at the very end of
the paper.

\section{Fr\"oberg's result and its later refinements}\label{section_froberg}

Before stating Fr\"oberg's result, let us recall some definitions.

\begin{definition}\label{def_cycle}
 {\rm
Let $G$ be a graph with vertex set $V$ and edge set $E$. Given
$t\geq 3$, a {\it $t$-cycle} in $G$ is a subgraph of $G$ whose edges
are of the form $(v_1,v_2)$, $(v_2,v_3)$, \ldots, $(v_t,v_1)$ where
$v_1,\ldots,v_t$ are distinct elements in $V$. A $t$-cycle $C$ in
$G$ has a {\it chord} if $(v_i,v_j)\in E$ for some $1\leq i<j\leq t$
such that $(v_i,v_j)$ is not an edge of $C$. An {\it induced} (or
{\it minimal}) $t$-cycle is a $t$-cycle in $G$ with no chord. More
generally, a subgraph of $G$ is {\it induced} if its edge set
contains all the elements in $E$ joining two distinct elements in
its vertex set. A graph $G$ is said to be {\it chordal} if it has no
induced $t$-cycle with $t\geq 4$.
 }\end{definition}

\begin{theorem}[{\cite[Theorem~1]{Froberg}}]\label{thm_froberg}
An edge ideal $I$ has a linear resolution \iff{} the graph $G(I)^c$
is chordal.
\end{theorem}

This result has recently been recovered by Eisenbud, Green, Hulek
and Popescu. Moreover, if the resolution is nonlinear, they also
determine the step in the resolution where nonlinear syzygies first
appear:

\begin{theorem}[{\cite[Theorem~2.1]{EisEtal}}]\label{thm_eisetal}
If $I$ is an edge ideal with nonlinear resolution, the smallest
integer $r\geq 4$ such that $\beta_{r-3,j}\neq 0$ for some $j\geq r$
coincides with the smallest integer $r\geq 4$ such that $G(I)^c$ has
an induced $r$-cycle.
\end{theorem}

The general monomial case is usually reduced to the squarefree
monomial case via polarization but one can sometimes provide a
direct result. This is illustrated by the following result that
gathers two characterizations that have been obtained independently.
Recall that a homogeneous ideal is {\it linearly presented} when the
module of its first syzygies is generated by linear ones. We will
refer to a vertex $v$ such that the loop $(v,v)$ belongs to
$E_{G(I)}$ as a {\it square vertex} of $G(I)$. The {\it edge graph}
$G(I)^\ast$ of $G(I)$ featured in
Proposition~\ref{prop_linearpresent}~(\ref{item_in_linearpresent_polar})
has vertex set $V_{G(I)^\ast}:=\{(1),\ldots,(m)\}$ and edge set
$E_{G(I)^\ast}:=\{(i,j);\, 1\leq i<j\leq m\,/\ \lcm{f_i,f_j}\neq
1\}$. It is always a simple graph. The {\it distance} between two
vertices of $G(I)^\ast$ is the minimum length of a path connecting
them, and the {\it diameter} of $G(I)^\ast$ is the longest distance,
i.e., the longest shortest path, between any two of its vertices
(the diameter of a nonconnected graph is infinite).

\begin{proposition}\label{prop_linearpresent}
Given an ideal $I$ generated by monomials of degree two, the
following are equivalent:
\begin{enumerate}
\item\label{item_in_linearpresent_linearpresent}
$I$ is linearly presented.
\item\label{item_in_linearpresent_eisetal}
The edge ideal $I_{sq}$ obtained from $I$ by removing its square
generators is linearly presented, any two square vertices of $G(I)$
are adjacent, and for any edge  $(v_i,v_j)$ of $G(I)$ that is not a
loop and any square vertex $v_k$, $v_k$ is adjacent to either $v_i$
or $v_j$.
\item\label{item_in_linearpresent_polar}
The graph $G(I)^\ast$ has diameter $\leq 2$.
\end{enumerate}
\end{proposition}

\smallskip\noindent
(\ref{item_in_linearpresent_linearpresent}) $\iffm$
(\ref{item_in_linearpresent_eisetal}) is
\cite[Proposition~2.3~(a)]{EisEtal} and
(\ref{item_in_linearpresent_linearpresent}) $\iffm$
(\ref{item_in_linearpresent_polar}) is \cite[Lemma~4.28]{PolarSyz}.
Observe that using Theorem~\ref{thm_eisetal} in order to translate
the condition on $I_{sq}$ in (\ref{item_in_linearpresent_eisetal})
in terms of the graph $G(I_{sq})^c=G(I)^c$, one can easily check
directly that (\ref{item_in_linearpresent_eisetal}) $\iffm$
(\ref{item_in_linearpresent_polar}).

\medskip
Finally, when $I$ is linearly presented one has the following
result:

\begin{proposition}[{\cite[Proposition~2.3~(b)]{EisEtal}}]\label{prop_eisetal_nonsquarefree}
If $I$ is an ideal generated by monomials of degree two which is
linearly presented, let $I_{sq}$ be the edge ideal obtained from $I$
by removing its square generators. Then, $I$ has a linear resolution
\iff{} $I_{sq}$ has. Moreover, when this does not occur, the step of
the \mgfr{} where nonlinear syzygies first appear is the same for
$I$ and $I_{sq}$.
\end{proposition}

\begin{remark}{\rm
Theorem~\ref{thm_main} contains Theorem~\ref{thm_eisetal} and
Propositions~\ref{prop_linearpresent} and
\ref{prop_eisetal_nonsquarefree}:
\begin{itemize}
\item
In order to recover Theorem~\ref{thm_eisetal}, observe that if $I$
is an edge ideal, $G(I)$ has an induced subgraph consisting of two
disjoint edges \iff{} $G(I)^c$ has an induced $4$-cycle. Note that
this claim is wrong if we do not assume that $I$ is squarefree.
\item
It is easy to check that $G(I)$ has no induced subgraph consisting
of two disjoint edges \iff{} $G(I)^\ast$ has diameter $\leq 2$ and
hence, Theorem~\ref{thm_main}~(\ref{item_in_mainthm_linpres})
contains Proposition~\ref{prop_linearpresent}.
\item
Finally, observe that if we assume that $G(I)$ has no induced
subgraph consisting of two disjoint edges, i.e., that $I$ is
linearly presented by
Theorem~\ref{thm_main}~(\ref{item_in_mainthm_linpres}), then the
other combinatorial conditions in Theorem~\ref{thm_main} and also
the value of $i_3$ in
Theorem~\ref{thm_main}~(\ref{item_in_mainthm_nonlinpres}) only
depend on the graph $G(I)^c=G(I_{sq})^c$ and hence,
Proposition~\ref{prop_eisetal_nonsquarefree} follows.
\end{itemize}
 }\end{remark}

\section{An example where the whole Betti diagram is determined}\label{section_example}

We focus in this section on the case where $I$ is the edge ideal
generated by all the squarefree monomials of degree two in the
variables $x_1,\ldots,x_n$ except $x_1x_2$, $x_2x_3,\ \ldots$,
$x_nx_1$. Then, $G(I)^c$ is an $n$-cycle. In
\cite[Example~2.2]{EisEtal}, it is shown that the syzygies of $I$
are linear in steps $0,\ldots,n-4$ and that $\beta_{n-3,n}=1$. The
whole resolution of $I$ is given by the following result:

\begin{proposition}\label{prop_example}
Given $n\geq 4$, if $I\subset R:=\kx$ is an edge ideal such that
$G(I)^c$ is a cycle with $n$ vertices, the \mgfr{} of $I$ is
$$
0\lar R(-n)\lar R(-n+2)^{\beta_{n-4}}\lar\cdots\lar
R(-2)^{\beta_{0}}\lar I\lar 0
$$
where for all $i$, $0\leq i\leq n-4$,
$\beta_{i}:=n\frac{i+1}{n-i-2}\binom{n-2}{i+2}$.
\end{proposition}

\begin{remark}{\rm
Using the Auslander-Buchsbaum formula, the above result implies that
$R/I$ is a Gorenstein ring of dimension two. The well-know symmetry
of the Betti numbers when $R/I$ is Gorenstein can be observed
checking easily in our formula that $\beta_{n-4-i}=\beta_i$ for all
$i$, $0\leq i\leq n-4$.
 }\end{remark}

\medskip
In order to prove Proposition~\ref{prop_example}, we use a special
case of Hochster's formula (see \cite{Hochster}) for an edge ideal
as stated in \cite[Proposition~1.2]{RothVantuyl}: denoting by
$V:=\{(1),\ldots,(n)\}$ the vertex set of $G(I)$,
\begin{equation}\label{eq_vantuyl_1}
\beta_{i,j}=
 \sum_{S\subseteq V\,;\ |S|=j}{\rm
dim}_K\tilde{H}_{j-i-2}(\Delta(G_S^c),K)\ ,\quad\forall i,j\geq 0\,,
\end{equation}
where $G_S$ denotes the induced subgraph of $G(I)$ on the vertex set
$S$, $G_S^c$ is its complement (in the vertex set $S$), and
$\Delta(G_S^c)$ is the clique complex of $G_S^c$. Note that
$\Delta(G_S^c)$ is a subcomplex of $\Delta:=\Delta(G(I)^c)$, the
clique complex of $G(I)^c$. Indeed, $\Delta(G_S^c)$ coincides with
$\Delta_S$, the subcomplex of $\Delta$ whose vertex set is $S$.

\medskip
Since $G(I)^c$ is a cycle, the only subcomplex of $\Delta$ with
nontrivial homology in degree $>0$ is $\Delta$ itself which reduced
homology is $K$ and is concentrated in degree 1. Thus,
$\beta_{n-3,n}=1$ and $\beta_{i,j}=0$ for any other $i,j$ such that
$j>i+2$.

\medskip
One can now determine the first row of the Betti diagram using the
formulation of (\ref{eq_vantuyl_1}) given in
\cite[Proposition~2.1]{RothVantuyl} when $j=i+2$:
\begin{equation}\label{eq_vantuyl_2}
\beta_{i,i+2}=
 \sum_{S\subseteq V\,;\ |S|=i+2}(\#\comp{G_S^c}-1)\ ,\quad\forall i\geq
 0\,.
\end{equation}
Subgraphs with one component have no contribution in
(\ref{eq_vantuyl_2}) so $\beta_{i,i+2}=0$ if $i+2\geq n-1$. Since
the number $k$ of components of $G_S^c$ satisfies $1\leq k\leq i+2$
for all $S\subseteq V$ with $|S|=i+2$, one has that for all $i\geq
0$ such that $i+2<n-1$,
\begin{equation}\label{eq_vantuyl_3}
\beta_{i,i+2}
 =
 \sum_{k=2}^{i+2}(\sum_{
 {\scriptsize
 \begin{array}{c}
 S\subseteq V\,;\\ |S|=i+2\ {\rm and}\\
 \#\comp{G_S^c}=k
 \end{array}
 }}(k-1))
 =
 \sum_{k=2}^{i+2}(k-1)\,N(i+2,k)
\end{equation}
where $N(i+2,k)$ is the number of induced subgraphs of $G_S^c$ with
$i+2$ vertices and $k$ components. We need the following technical
lemma:

\begin{lemma}\label{lem_numero_subgrafos_en_ciclo}
Let $i,k,n$ be integers such that $0<k\leq i<n$ and let $C$ be an
$n$-cycle. Then, the number of induced subgraphs of $C$ with $i$
vertices and $k$ components is
$\frac{n}{k}\binom{i-1}{k-1}\binom{n-i-1}{k-1}$.
\end{lemma}

\begin{proof}
Let $V:=\{(1),\ldots,(n)\}$ be the vertex set of $C$. Given a subset
$S$ of $V$, the induced subgraph $C_S$ of $C$ with vertex set $S$
can be represented by a vector $w_S$ of length $n$ whose $\ell$th
entry is $1$ if $(\ell)\in S$ and $0$ otherwise. From now on, we
identify the induced subgraph $C_S$ of $C$ to the vector $w_S$.

\smallskip
This identification can be used to compute the number of induced
subgraphs of $C$ with $i$ vertices and $k$ components. Indeed, the
number $i$ of vertices in $w_S$ is the number of its nonzero
entries, and the number $k$ of components of $w_S$ can easily be
related to the number of blocks of nonzero entries in $w_S$. In
order to avoid distinguishing cases as when the vector $w_S$
starts/ends with $1$/$0$, we make an easy observation. Consider the
set $W$ of vectors $w$ of length $n$ with entries 0 and 1, whose
first entry is 1 and last entry is 0, with $i$ nonzero entries and
$k$ blocks of nonzero entries (hence $k$ blocks of zero entries). To
each $w$ in $W$, we can associate $n$ subgraphs of $C$ with $i$
vertices and $k$ components assigning to the first entry of $w$ one
of the vertices in $\{(1),\ldots,(n)\}$. Conversely, each induced
subgraph $w_S$ of $C$ with $i$ vertices and $k$ components always
comes from $k$ vectors $w$ in $W$ depending on which of the blocks
of nonzero entries of $w_S$ is the first block of $w$. This implies
that the number of induced subgraphs of $C$ with $i$ vertices and
$k$ components is
 $\frac{n}{k}\times |W|$.
It is an easy exercise to show that the number of elements in $W$ is
equal to $\binom{i-1}{k-1}\binom{n-i-1}{k-1}$ and the result
follows.
\end{proof}

\medskip
Applying Lemma~\ref{lem_numero_subgrafos_en_ciclo} in
(\ref{eq_vantuyl_3}), one gets that, for all $i<n-3$,
$$
\beta_{i,i+2}
 =
\sum_{k=2}^{i+2}(k-1) \frac{n}{k}\binom{i+1}{k-1}\binom{n-i-3}{k-1}
 =
 n\sum_{k=1}^{i+1}
 \frac{k}{k+1}\binom{i+1}{k}\binom{n-i-3}{k}
$$
and Proposition~\ref{prop_example} now follows applying the
following combinatorial lemma to $m=n-2$ and $a=i+1$:

\begin{lemma}\label{lem_numeros_combinatorios}
For any two integers $m$ and $a$ such that $1\leq a<m$,
$$
\sum_{k=1}^a\frac{k}{k+1}\binom{m-a}{k}\binom{a}{k}=\frac{a}{m-a+1}\binom{m}{a+1}\
.
$$
\end{lemma}

\begin{proof}
Let $F$ and $g$ be the two polynomials in $\Q[X]$ defined as
follows:
$$
F:=(1+X)^a=\sum_{k=0}^{a}\binom{a}{k}X^k
 \quad\hbox{and}\quad
g:=(1+X)^{m-a}=\sum_{k=0}^{m-a}\binom{m-a}{k}X^k\ .
$$
Set $f:=F'$ and $G:=\int_0^xg(u)du$. Then
 $\displaystyle{
f=a(1+X)^{a-1}=\sum_{k=1}^a k \binom{a}{k}X^{k-1}=\sum_{k=1}^a k
\binom{a}{k}X^{a-k}
 }$
where the last equality follows from the fact that $X^{k-1}$ and
$X^{a-k}$ have the same coefficients in the polynomial $\sum_{k=1}^a
k \binom{a}{k}X^{k-1}$ because
$k\binom{a}{k}=k\frac{a!}{k!(a-k)!}=a\binom{a-1}{k-1}$ and
$(a-k+1)\binom{a}{a-k+1}=(a-k+1)\frac{a!}{(a-k+1)!(k-1)!}=a\binom{a-1}{k-1}$.
On the other hand,
 $\displaystyle{
G=\frac{(1+X)^{m-a+1}-1}{m-a+1}=}$ $\displaystyle{\sum_{k=0}^{m-a}
\frac{1}{k+1} \binom{m-a}{k}X^{k+1}
 }$.

 \smallskip
Expressing the polynomial $f G$ in two different ways, one gets that
$$
\frac{a((1+X)^{m}-(1+X)^{a-1})}{m-a+1}=(\sum_{k=1}^a k
\binom{a}{k}X^{a-k})(\sum_{k=0}^{m-a} \frac{1}{k+1}
\binom{m-a}{k}X^{k+1})
$$
and the expected formula now follows determining the coefficient of
$X^{a+1}$ in both sides of this equality.
\end{proof}

\section{Nonlinear syzygies of smallest degree}\label{section_results}

In this section, we compute the number of nonlinear syzygies of $I$
of smallest degree when $I$ does not have a linear resolution.
Observe that nonlinear syzygies have degree at least $4$. The case
where $I$ has nonlinear syzygies of degree $4$ happens to be special
when $I$ has square generators as we shall observe in
Remark~\ref{rk_ris4_notsquarefree}. So let us start characterizing
when $I$ has nonlinear syzygies of degree $4$ and computing the
number of such syzygies when this occurs. It is a direct consequence
of the following result that has be proved by Katzman in
\cite[Lemma~2.2]{Katzman} when $I$ is squarefree.

\begin{lemma}\label{lem_katzman_with_loops}
Let $I$ be an ideal generated by monomials of degree two. For all
$i\geq 0$, $\beta_{i,2(i+1)}$ is the number of induced subgraphs of
$G(I)$ consisting of $i+1$ disjoint edges.
\end{lemma}

\begin{proof}
If $I$ is not squarefree, one can assume without loss of generality
that the square generators of $I$ are $x_1^2,\ldots,x_p^2$ for some
$p\leq n$. The polarization $I'$ of $I$ is the edge ideal obtained
by substituting, for all $j\in\{1,\ldots,p\}$, $x_{j}x_{n+j}$ for
$x_{j}^2$ where $x_{n+1},\ldots,x_{n+p}$ are $p$ new variables. Note
that the graph $G(I')$ associated to $I'$ is obtained from $G(I)$ by
substituting each loop $(j,j)$ in $G(I)$ ($1\leq j\leq p$) by the
whisker $(j,n+j)$, and two disjoint edges of $G(I')$ form an induced
subgraph of $G(I')$ \iff{} the corresponding edges of $G(I)$ form an
induced subgraph of $G(I)$. The result now follows applying
\cite[Lemma~2.2]{Katzman} to the edge ideal $I'$ and using the
well-known fact that the graded Betti numbers of $I$ (as
$K[x_1,\ldots, x_n]$-module) and $I'$ (as $K[x_1,\ldots,
x_n,x_{n+1},\ldots,x_{n+p}]$-module) are the same.
\end{proof}

\begin{corollary}\label{cor_katzman}
If $I$ is an ideal generated by monomials of degree two,
$\beta_{1,4}\neq 0$ \iff{} $G(I)$ has at least one induced subgraph
consisting of two disjoint edges. When this occurs,
$$
\beta_{1,4}= \#\{\hbox{induced induced subgraphs of $G(I)$
consisting of 2 disjoint edges}\}\,.
$$
\end{corollary}

\medskip
We can now start to prove Theorem~\ref{thm_lastthm} that gives the
number of nonlinear syzygies of $I$ of smallest degree (except in
the already studied case where $I$ is not squarefree and has
nonlinear syzygies of degree 4). We shall use essentially
Theorem~\ref{thm_froberg}, Proposition~\ref{prop_example} and a
result by Gasharov, Hibi and Peeva that we recall below. For all
$\underline{s}=(s_1,\ldots,s_n)\in\N^n$, set
$\underline{x}^{\underline{s}}:=x_1^{s_1}\cdots x_n^{s_n}\in R$.
Consider a minimal $\N^n$-graded free resolution of $I$, ${\mathbf
F}$, and given any monomial $\underline{x}^{\underline{s}}$ in $R$,
denote by ${\mathbf F}_{\underline{s}}$ the subcomplex of ${\mathbf
F}$ that is generated by the $\N^n$-homogeneous basis elements of
degrees dividing $\underline{x}^{\underline{s}}$.

\begin{theorem}[{\cite[Theorem~2.1]{PeevaEtal}}]\label{thm_peevaetal}
For any monomial $\underline{x}^{\underline{s}}\in R$, ${\mathbf
F}_{\underline{s}}$ is a minimal $\N^n$-graded free resolution of
$I_{\underline{s}}$, the monomial ideal generated by $\{f_i;\ f_i\
{\rm divides}\ \underline{x}^{\underline{s}}\}$.
\end{theorem}

Together with Proposition~\ref{prop_example}, this result implies
the following result:

\begin{proposition}\label{prop_ineq_betti} Assume that $G(I)^c$ has an induced
$r$-cycle with $r\geq 4$, and consider the vertex set
$\{(i_1),\ldots,(i_r)\}$ of this cycle. Let $\underline{s}$ be the
element in $\N^n$ such that
$\underline{x}^{\underline{s}}=x_{i_1}\cdots x_{i_r}$. Then, the
multigraded Betti numbers $\beta_{i,\underline{s}}$ are equal to $0$
for all $i<r-3$, and $\beta_{r-3,\underline{s}}$ is equal to $1$. In
particular,
$$
\beta_{r-3,r}\geq\#\{\hbox{induced $r$-cycles in $G(I)^c$}\}\,.
$$
\end{proposition}

\begin{proof}
Applying Theorem~\ref{thm_peevaetal} to
$\underline{x}^{\underline{s}}=x_{i_1}\cdots x_{i_r}$, one gets that
the subcomplex ${\mathbf F}_{\underline{s}}$ of ${\mathbf F}$ is a
minimal $\N^n$-graded free resolution of $I_{\underline{s}}$. In
particular, the multigraded Betti number $\beta_{i,\underline{s}}$
coincides with the number of generators of multidegree
$\underline{s}$ in the $i$th step of the \mgfr{} of
$I_{\underline{s}}$ for all $i$. Now observe that
$I_{\underline{s}}\subset R$ is minimally generated by all the
square free monomials of degree two in $x_{i_1},\ldots,x_{i_r}$
except $r$ of them that correspond to the edges in the induced
$r$-cycle of $G(I)^c$. Since the minimal generators of
$I_{\underline{s}}$ only involve the variables
$x_{i_1},\ldots,x_{i_r}$, the multigraded Betti numbers of
$I_{\underline{s}}$ (as $R$-module) are the same as the multigraded
Betti numbers of $I_{\underline{s}}\cap K[x_{i_1},\ldots,x_{i_r}]$
(as $K[x_{i_1},\ldots,x_{i_r}]$-module), and the result follows from
Proposition~\ref{prop_example} where the resolution of
$I_{\underline{s}}\cap K[x_{i_1},\ldots,x_{i_r}]\subset
K[x_{i_1},\ldots,x_{i_r}]$ is described.
\end{proof}

\begin{remark}{\rm
Since we have described all the Betti numbers in
Proposition~\ref{prop_example}, the above argument carries some
additional information on the graded Betti numbers of $I$. Indeed,
for all $r\geq 4$ such that $G(I)^c$ has an induced $r$-cycle, and
for all $i$ such that $0\leq i\leq r-4$, $\beta_{i,i+2}\geq
r\frac{i+1}{r-i-2}\binom{r-2}{i+2}$.
 }\end{remark}

\begin{theorem}\label{thm_lastthm}
Assume that $I$ does not have a linear resolution, and let $r$ be
the smallest integer {\rm(}$\geq 4${\rm)} such that $\beta_{i,r}\neq
0$ for some $i\leq r-3$. If $r=4$, assume moreover that $I$ is
squarefree. Then, $r$ is the smallest integer such that $G(I)^c$ has
an induced $r$-cycle and
$$
\beta_{r-3,r}=\#\{\hbox{induced $r$-cycles in $G(I)^c$}\}\,.
$$
\end{theorem}

\begin{proof}
Consider $\underline{s}\in\N^n$ with $|\underline{s}|=r$ such that
$\beta_{i,\underline{s}}>0$ for some $i\leq r-3$.

\smallskip
We first assume that $\underline{x}^{\underline{s}}=x_{i_1}\cdots
x_{i_r}$ with $i_1,\ldots,i_r$ all different. Then the ideal
$I_{\underline{s}}$ is squarefree. Moreover, it has a nonlinear
resolution by Theorem~\ref{thm_peevaetal}. Applying
Theorem~\ref{thm_froberg} to the ideal $J:=I_{\underline{s}}\cap
K[x_{i_1},\ldots, x_{i_r}]$, one gets that the complement $G(J)^c$
of $G(J)$ (in the vertex set $\{(i_1),\ldots,(i_r)\}$) has an
induced $\ell$-cycle for some $\ell\leq r$. This provides an induced
$\ell$-cycle in $G(I)^c$. By Proposition~\ref{prop_ineq_betti},
$\beta_{\ell-3,\ell}>0$ and $\ell\geq r$ by minimality of $r$. Thus
$\ell=r$, and we have shown that if $\underline{s}\in\N^n$ is such
that $\beta_{i,\underline{s}}>0$ for some $i\leq r-3$ and
$\underline{x}^{\underline{s}}=x_{i_1}\cdots x_{i_r}$ is squarefree,
then $G(I)^c$ has an induced $r$-cycle with vertex set
$\{(i_1),\ldots,(i_r)\}$. By Proposition~\ref{prop_ineq_betti}, this
implies that $\beta_{i,\underline{s}}=0$ for all $i<r-3$ and
$\beta_{r-3,\underline{s}}=1$.

\smallskip
Assume now that the monomial $\underline{x}^{\underline{s}}$ is not
squarefree and let us see that this leads to a contradiction. If
$I_{\underline{s}}$ was squarefree, one could apply the same
argument as before and find an induced $\ell$-cycle in $G(I)^c$ with
$\ell\geq r$ by minimality of $r$. On the other hand, this cycle
should be supported on vertices corresponding to variables dividing
the monomial $\underline{x}^{\underline{s}}$ which is not squarefree
and has degree $r$, and hence $\ell<r$. Thus, $I_{\underline{s}}$ is
not squarefree. In particular, $I$ is not squarefree and hence
$r\geq 5$ by hypothesis. Consider the ideal
$J:=I_{\underline{s}}\cap K[x_{i_1},\ldots, x_{i_q}]$ where
$\{i_1,\ldots,i_q\}$ is the support of the monomial
$\underline{x}^{\underline{s}}$, and assume without loss of
generality that the variables $x_i$ with $i\in\{i_1,\ldots,i_q\}$
such that $x_i^2$ is a generator of $J$ are $x_{i_1},\ldots,x_{i_p}$
($p\leq q$). The polarization $J'$ of $J$ is the edge ideal obtained
by substituting, for all $j\in\{1,\ldots,p\}$, $x_{i_j}x_{n+j}$ for
$x_{i_j}^2$ where $x_{n+1},\ldots,x_{n+p}$ are $p$ new variables.
Since the graded Betti numbers of $J$ (as $K[x_{i_1},\ldots,
x_{i_q}]$-module) and $J'$ (as $K[x_{i_1},\ldots,
x_{i_q},x_{n+1},\ldots,x_{n+p}]$-module) are the same, by applying
Theorem~\ref{thm_froberg} to the ideal $J'$, one gets that the
complement $G(J')^c$ of the simple graph $G(J')$ has an induced
$\ell$-cycle for some $\ell\leq p+q$. Denote this $\ell$-cycle by
$C$. Note that the simple graph $G(J')$ associated to $J'$ is
obtained from $G(J)$ by substituting each loop $(i_j,i_j)$ ($1\leq
j\leq p$) by the whisker $(i_j,n+j)$. We shall show in
Lemma~\ref{lem_cycles_and_whiskers} below that no induced $t$-cycle
of $G(J')^c$ with $t\geq 5$ passes through any of the vertices
$(n+1)$, \ldots, $(n+p)$. This implies that if $\ell\geq r\geq 5$,
then the vertex set of $C$ should be contained in
$\{(i_1),\ldots,(i_q)\}$, and hence $\ell\leq q<r$, a contradiction.
Thus $\ell$ has to be strictly smaller than $r$ but, by
Proposition~\ref{prop_ineq_betti}, this implies that
$\beta_{\ell-3,\ell}>0$ and by minimality of $r$, one also gets a
contradiction.

\smallskip
We have shown that if $\underline{s}\in\N^n$ is such that
$|\underline{s}|=r$ and $\beta_{i,\underline{s}}>0$ for some $i\leq
r-3$, then the monomial $\underline{x}^{\underline{s}}=x_{i_1}\cdots
x_{i_r}$ is squarefree and $G(I)^c$ has an induced $r$-cycle with
vertex set $\{(i_1),\ldots,(i_r)\}$. Moreover,
$\beta_{i,\underline{s}}=0$ for all $i<r-3$ and
$\beta_{r-3,\underline{s}}=1$, and the result follows using
Proposition~\ref{prop_ineq_betti}.
\end{proof}

\begin{lemma}\label{lem_cycles_and_whiskers}
Let $G$ be a simple graph and $v$ a vertex of $G$ of order 1. Then,
no induced $t$-cycle of the complement $G^c$ of $G$ has $v$ in its
vertex set if $t\geq 5$.
\end{lemma}

\begin{proof}
Let $v'$ be the vertex of $G$ adjacent to $v$. Consider $S$, a
subset with $t$ elements of the vertex set of $G$ containing $v$.
The order of the vertex $v$ in the induced subgraph $G_S^c$ of $G^c$
with vertex set $S$ is either $t-2$ if $v'\in S$, or $t-1$
otherwise. Since $t-1$ and $t-2$ are strictly bigger that $2$ if
$t\geq 5$ and the order of any vertex in a cycle is $2$, the result
follows.
\end{proof}

\begin{remark}\label{rk_ris4_notsquarefree}{\rm
If $I$ is an edge ideal and $r=4$, one can easily check that the
value of $\beta_{1,4}$ obtained in Theorem~\ref{thm_lastthm}
coincides with the one in Corollary~\ref{cor_katzman}. Nevertheless,
when $I$ is not squarefree and $r=4$, the correct value of
$\beta_{1,4}$ is the one given in Corollary~\ref{cor_katzman} and
the formula in Theorem~\ref{thm_lastthm} may be wrong as the
following example shows: if we consider the ideal
$I=(x_1^2,x_1x_3,x_3x_5,x_5x_2,x_2x_4,x_4x_1)\subset\Q[x_1,\ldots,x_5]$,
the complement $G(I)^c$ of $G(I)$ has no induced $4$-cycles while
$\beta_{1,4}=1$. The Betti diagram can be obtained using
\cite{Macaulay} or any other program devoted to computations in
algebraic geometry and commutative algebra like {\sc Singular}, {\sc
CoCoA} or {\sc Macaulay2}. The unique induced subgraph of $G(I)$
that makes a contribution to $\beta_{1,4}$ in the value given by
Lemma~\ref{lem_katzman_with_loops} is the one formed by the loop and
the horizontal edge in its representation below.

\bigskip

\begin{center}
\begin{picture}(55,50)(0,-20)
\begin{tabular}{|r|cccc|}\hline &0&1&2&3\\\hline 2&6&7&1&-\\
3&-&1&3&1 \\\hline
\end{tabular}
\put(-80,-30){\tiny{The Betti diagram of $I$.}}
\end{picture}
\hskip 2.5cm
\begin{picture}(75,70)(0,-5)
 \put(35,60){\circle{10}}
 \put(35,55){\circle*{4}}
 \put(5,33){\circle*{4}}
 \put(65,33){\circle*{4}}
 \put(16,0){\circle*{4}}
 \put(54,0){\circle*{4}}
 \put(17,0){\line(1,3){18}}
 \put(17,0){\line(3,2){48}}
 \put(53,0){\line(-1,3){18}}
 \put(53,0){\line(-3,2){48}}
  \put(5,33){\line(1,0){61}}
  \put(-22,-15){\tiny{The graph $G(I)$ associated to $I$.}}
\end{picture}
\hskip 2cm
\begin{picture}(75,70)(0,-5)
 \put(35,54){\circle*{4}}
 \put(5,33){\circle*{4}}
 \put(65,33){\circle*{4}}
 \put(17,0){\circle*{4}}
 \put(53,0){\circle*{4}}
\put(5,33){\line(3,2){30}}
 \put(65,33){\line(-3,2){30}}
\put(53,0){\line(1,3){11}} \put(17,0){\line(-1,3){11}}
\put(17,0){\line(1,0){36}} \put(-17,-15){\tiny{The complement
$G(I)^c$ of $G(I)$.}}
\end{picture}
\end{center}
 }\end{remark}

 \bigskip

\begin{remark}{\rm
The proof of Theorem~\ref{thm_lastthm} carries some additional
information: a multidegree $\underline{s}\in\N^n$ contributes to
$\beta_{r-3,r}$ \iff{} $\underline{x}^{\underline{s}}$ is the
product of the vertices involved in an induced $r$-cycle of
$G(I)^c$. Moreover, for such an $\underline{s}$, the corresponding
multigraded Betti number $\beta_{r-3,\underline{s}}$ is equal to 1.
  }\end{remark}

 \section{The shape of the Betti diagram}\label{section_shape}

The only nonzero entry on the first column of the Betti diagram of
$I$ is $\beta_{0,2}=\#\,E_{G(I)}$. The following result of Katzman
provides some information on the shape of the Betti diagram when $I$
is an edge ideal.

\begin{lemma}[{\cite[Lemma~2.2]{Katzman}}]\label{lem_katzman_shape}
For any edge ideal, $\beta_{i,j}=0$ for all $i\geq 0$ and
$j>2(i+1)$.
\end{lemma}

In terms of the Betti diagram, this means that the entries below the
diagonal through $\beta_{0,2}$ are zero. Our next result improves
this statement.

\begin{theorem}\label{thm_shape_Betti}
Consider an ideal $I\subset \kx$ generated by monomials of degree
two, with $K$ an arbitrary field, and let $m$ be its
Castelnuovo-Mumford regularity. For all $d\geq 3$, set
$\displaystyle{i_d:=\min_{1\leq i\leq p}{\{i;\ \beta_{i,i+d}\neq
0\}}}$ if $\beta_{i,i+d}\neq 0$ for some $i\geq 1$, $i_d:=0$
otherwise. If $I$ does not have a linear resolution, then
 $$ 1\leq i_3<\cdots<i_m\leq p\,.$$
\end{theorem}

Now we know that the shape of the Betti diagram of an ideal
associated to a graph is:

{\small
$$
\begin{array}{|c|ccccccccccccc|}\hline
 &0&1&\ldots&i_3-1&i_3&\ldots&i_4-1&i_4&\ldots&i_m-1&i_m&\ldots&p
 \\\hline
2&\bullet&\ast&\cdots&\ast&\ast&\cdots&\ast&\ast&\cdots&\ast&\ast&\cdots&\ast
 \\
3&-&-&\cdots&-&\bullet&\cdots&\ast&\ast&\cdots&\ast&\ast&\cdots&\ast
 \\
4&-&-&\cdots&-&-&\cdots&-&\bullet&\cdots&\ast&\ast&\cdots&\ast
 \\
\vdots&&&&&&&&&&&&&
 \\
m&-&-&\cdots&-&-&\cdots&-&-&\cdots&-&\bullet&\cdots&\ast
 \\\hline
\end{array}
$$
\begin{center}
$\bullet=$ nonzero entry\,; $-=$ zero entry\,; $\ast=$ entry that
may be zero or not.
\end{center}
}

\bigskip
This result will be a direct consequence of our last lemma. Recall
that given a simplicial complex $\Delta$ and a vertex $x$ of
$\Delta$, the subcomplex $\link{x,\Delta}$ of $\Delta$ is defined by
$$\link{x,\Delta}:= \{F\in \Delta ;\, x\not\in F\ {\rm
and}\ F\cup\{x\}\in \Delta\}\,.$$

\begin{lemma}\label{lem_shape}
Assume that $I$ is an edge ideal with nonlinear resolution. Using
the notations in Theorem~\ref{thm_shape_Betti}, if $i_d\geq 1$ for
some $d\geq 3$, then there exists $i<i_d$ such that
$\beta_{i,i+d-1}\neq 0$.
\end{lemma}

\begin{proof}
Let $V$ be the vertex set of the graph $G(I)$, and consider $d$,
$3\leq d\leq m$, such that $i_d\geq 1$. Then, $\beta_{i_d,i_d+d}\neq
0$ and $\beta_{i,i+d}=0$ for all $i$, $0\leq i<i_d$. Using
Hochster's formula as stated in (\ref{eq_vantuyl_1}), we have that
there exists $S\subseteq V$ with $|S|=i_d+d$ such that
$\widetilde{H}_{d-2}(\Delta(G_S^c),K)\neq 0$ while
$\widetilde{H}_{d-2}(\Delta(G_{S'}^c),K)=0$ for any subset
$S'\subset V$ with $|S'|<|S|$.

\smallskip
For $x\in S$, denote by $S_x$ the vertex set of
$\link{x,\Delta(G_{S}^c)}$. We claim that there exists a vertex
$x\in S$ such that $|S_x|<|S\setminus\{x\}|$. Otherwise, every pair
of vertices in $S$ would be linked in $\Delta(G_{S}^c)$ and hence
$\Delta(G_{S}^c)$ would be an $(i_d+d)$-simplex. Thus,
$\widetilde{H}_{j}(\Delta(G_S^c),K)=0$ for all $j\neq i_d+d-2
>d-2$, a contradiction with the fact that
$\widetilde{H}_{d-2}(\Delta(G_S^c),K)\neq 0$.

\smallskip
So we can choose a vertex $x\in S$ such that
$|S_x|<|S\setminus\{x\}|$ and consider the following long exact
sequence introduced in \cite{hibi},
$$
\cdots\rar\widetilde{H}_{d-3}(\link{x,\Delta(G_{S}^c),K}\rar\widetilde{H}_{d-2}(\Delta(G_{S}^c),K)\rar
\widetilde{H}_{d-2}(\Delta(G_{S}^c)_{S\setminus\{x\}},K)\rar\cdots
$$
where $\Delta(G_{S}^c)_{S\setminus\{x\}}$ is the subcomplex of
$\Delta(G_{S}^c)$ whose vertex set is $S\setminus\{x\}$. Note that
$\Delta(G_{S}^c)_{S\setminus\{x\}}=\Delta(G_{S\setminus\{x\}}^c)$
and so $\widetilde{H}_{d-2}(\Delta(G_{S}^c)_{S\setminus\{x\}},K)=0$
because $|S\setminus\{x\}|<|S|$. Since
$\widetilde{H}_{d-2}(\Delta(G_S^c),K)\neq 0$, one gets that
$\widetilde{H}_{d-3}(\link{x,\Delta(G_{S}^c)},K)\neq 0$.

\smallskip
Since we have chosen $x\in S$ such that
$|S_x|<|S\setminus\{x\}|=i_d+d-1$ and observing that
$\link{x,\Delta(G_{S}^c)}=\Delta(G^c_{S_x})$, we get that
$\widetilde{H}_{d-3}(\Delta(G_{S_x}^c),K)\neq 0$ with $j:=|S_x|\leq
i_d+d-2$. Using again Hochster's formula, we deduce that
$\beta_{j-d+1,j}>0$. Setting $i:=j-d+1\leq i_d+d-2-d+1=i_d-1<i_d$,
one has that $\beta_{i,i+d-1}>0$ and we are done.
\end{proof}

\begin{proof}[Proof of {Theorem~\ref{thm_shape_Betti}}]
The result trivially holds if $m=3$ and since a monomial ideal and
its polarization have the same Betti diagram, let us assume that
$m\geq 4$ and that $I$ is an edge ideal. Starting with $d=m$ and
applying Lemma~\ref{lem_shape}, one gets that there is a nonzero
entry on the $(d-1)$th row, i.e., $i_{d-1}\geq 1$, and that
$i_{d-1}<i_d$. Iterating the argument until $d=4$, the result
follows.
\end{proof}

\begin{remark}{\rm
\begin{enumerate}
\item
As observed in \cite[Section~4]{Katzman}, the Betti diagram (in
particular the number of rows and columns) may depend on the
characteristic of the field $K$. Nevertheless, it will have the
shape described in Theorem~\ref{thm_shape_Betti} in any
characteristic.
\item
By Theorem~\ref{thm_main}, $i_3$ and $\beta_{i_3,i_3+3}$ do not
depend on the characteristic of $K$. The four graphs given in
\cite[Appendix~A]{Katzman} show that the characteristic of $K$ may
be relevant for the computation of $\beta_{i_4,i_4+4}$.
\end{enumerate}
 }\end{remark}

All the ingredients are now available to prove the result stated in
the introduction:

\begin{proof}[Proof of Theorem~\ref{thm_main}]

(\ref{item_in_mainthm_linpres}): it is Corollary~\ref{cor_katzman}.

(\ref{item_in_mainthm_nonlinpres}): Assume that $i_3>1$. Then $G(I)$
has no induced subgraph consisting of two disjoint edges by
(\ref{item_in_mainthm_linpres}), Theorem~\ref{thm_lastthm} implies
that $G(I)^c$ has an induced cycle of length $\geq 5$ and gives the
values of $i_3$ and $\beta_{i_3,i_3+3}$, and the last statement is a
consequence of Theorem~\ref{thm_shape_Betti}. Conversely, if $G(I)$
has no induced subgraph consisting of two disjoint edges then
$i_3\neq 1$ by (\ref{item_in_mainthm_linpres}), and if $G(I)^c$ has
an induced cycle of length $\geq 5$ then $i_3\neq 0$ by
Proposition~\ref{prop_ineq_betti}.

(\ref{item_in_mainthm_linres}): The `if and only if' follows from
(\ref{item_in_mainthm_linpres}) and
(\ref{item_in_mainthm_nonlinpres}), and the last statement is a
consequence of Theorem~\ref{thm_shape_Betti}.
\end{proof}

\section*{Acknowledgements}
Research partially supported by MTM2007-61444, Ministerio de
Educaci\'on y Ciencia (Spain). The authors thank Isabel Bermejo and
Aron Simis for helpful conversations.

\bigskip\bigskip\noindent

\end{document}